\documentclass[11pt]{amsart}

\title[Topological symmetry groups and mapping class groups]
{Topological symmetry groups and mapping class groups for spatial graphs}

\author{Sangbum Cho}\thanks{The first-named author is supported in part by Basic Science Research Program through the National Research Foundation of Korea funded by the Ministry of Education, Science and Technology (2012R1A1A1006520).}
\address{
Department of Mathematics Education, Hanyang University, Seoul 133-791,
Korea}
\email{scho@hanyang.ac.kr}

\author{Yuya Koda}\thanks{The second-named author is supported in part by
Grant-in-Aid for Young Scientists (B) (No. 20525167), Japan Society for the Promotion of Science.} 

\address{Mathematical Institute, Tohoku University, Sendai, 980-8578, Japan}
\email{koda@math.tohoku.ac.jp}

\usepackage{amsfonts,amsmath,amssymb,amscd}

\usepackage{amsthm}
\usepackage{latexsym}
\usepackage[dvips]{graphicx}
\usepackage[dvips]{psfrag}
\usepackage[dvips]{color}

\theoremstyle{plain}
\newtheorem*{theorem*}{Theorem}
\newtheorem*{lemma*} {Lemma}
\newtheorem*{corollary*} {Corollary}
\newtheorem*{proposition*}{Proposition}
\newtheorem*{conjecture*}{Conjecture}
\newtheorem{theorem}{Theorem}[section]
\newtheorem{lemma}[theorem]{Lemma}
\newtheorem{corollary}[theorem]{Corollary}
\newtheorem{proposition}[theorem]{Proposition}

\theoremstyle{remark}

\newtheorem*{remark}{Remark}
\newtheorem*{definition}{Definition}
\newtheorem*{notation}{Notation}

\newtheorem*{question}{Question}

\theoremstyle{definition}

\newcommand{\Integer}{\mathbb{Z}}

\newcommand{\Int}{\mathrm{Int}}

\newcommand{\MCG}{\mathcal{MCG}}

\newcommand{\Homeo}{\mathrm{Homeo}}
\newcommand{\Aut}{\mathrm{Aut}}

\newcommand{\TSG}{\mathrm{TSG}}

\newcommand{\rel}{\mathrm{rel}}
\newcommand{\id}{\mathrm{id}}

\newcommand{\Fix}{\mathrm{Fix}}



\begin{document}
\maketitle

\begin{abstract}
We give a necessary and sufficient condition for
the mapping class group of the pair of
the 3-sphere and a graph embedded in it to be
isomorphic to
the topological symmetry group of the embedded graph.

\end{abstract}

\vspace{1em}

\begin{small}
\hspace{2em}  \textbf{2010 Mathematics Subject Classification}:
57M15, 05C10; 57M25, 05C25, 57S05


\hspace{2em} 
\textbf{Keywords}:
topological symmetry group; mapping class group; spatial graph; group of 

\hspace{2em} homeomorphisms; tunnel number one knot; tunnel.
\end{small}

\section*{Introduction}

By a {\it graph}
we shall mean the underlying
space of a finite connected simplicial complex of dimension one.
A {\it spatial graph} is a graph embedded
in a $3$-manifold.
The theory of
spatial {graphs} is
a generalization of classical knot theory.
For a spatial graph $\Gamma$ in $S^3$,
the mapping class group $\MCG (S^3 , \Gamma)$
($\MCG_+ (S^3 , \Gamma)$, respectively)
is defined to be the group of isotopy classes
of the  self-homeomorphisms
(orientation-preserving  self-homeomorphisms,
respectively) of $S^3$ that preserve $\Gamma$ setwise.
The cardinality of the group describes how many
symmetries the spatial graph admits.
In \cite{Kod} it is shown that
the group $\MCG(S^3, \Gamma)$ is always
finitely presented.

On the other hand, Simon \cite{Sim86}
(see also \cite{FNPT05, FT11} for details)
introduced another concept, called the 
{\it topological symmetry group} of a 
spatial graph $\Gamma$ in $S^3$, 
denoted by
$\TSG(S^3, \Gamma)$, to describe the
symmetries of a spatial graph $\Gamma$ in $S^3$.
This group is defined to be the subgroup of
the automorphism group
of $\Gamma$ induced by homeomorphisms of the pair 
$(S^3, \Gamma)$. 
When we allow only orientation preserving
homeomorphisms, we get the
{\it positive topological symmetry group} $\TSG_+(S^3, \Gamma)$.

The aim of this paper is to provide complete
answers (Theorems \ref{thm:the mcg and the positive tsg}
and \ref{thm:involution and the symmetry group of a spatial graph})
to the following question:
\begin{question}
\label{question:topological symmetry groups and mapping class groups}
When is $\TSG (S^3, \Gamma)$ $(\TSG_+ (S^3, \Gamma)$, respectively$)$
isomorphic to $\MCG (S^3, \Gamma)$ $(\MCG_+ (S^3, \Gamma)$, respectively$)$?
\end{question}
We note that one of the answers 
for this question, i.e. Theorem \ref{thm:the mcg and the positive tsg}, 
implies that, if the group $\MCG_+ (S^3, \Gamma)$ 
is finite, then it is a finite subgroup of $SO(4)$ by \cite{FNPT05}.

\begin{notation}
Let $X$ be a subset of a given polyhedral space $Y$.
Throughout the paper, we will denote the interior of
$X$ by $\Int \thinspace X$.
We will use $N(X; Y)$ to denote a closed regular neighborhood of $X$ in $Y$.
If the ambient space $Y$ is clear from the context,
we denote it briefly by $N(X)$.
Let $M$ be a $3$-manifold.
Let $L \subset M$ be a submanifold with or without boundary.
When $L$ is of 1 or 2-dimension, we write
$E(L) = M \setminus \Int \thinspace N(L)$.
When $L$ is of 3-dimension, we write
$E(L) = M \setminus \Int \thinspace L$.
\end{notation}

\section{Mapping class groups and topological symmetry groups}

Throughout this paper, we will work in
the piecewise linear category.

\subsection{Mapping class groups}
\label{subsec:Mapping class groups}

Let $N$ be a possibly empty subspace
of a compact orientable  manifold $M$.
We will denote by
\[ \Homeo (M, N )~(\Homeo (M ~\rel~N),~\mbox{respectively}) \]
the space of
self-homeomorphisms of
$M$ preserving $N$ setwise (preserving $N$ pointwise, respectively).
We call the group
\[ \pi_0 (\Homeo (M, N ))~( \pi_0 (\Homeo (M ~\rel~ N)),
~\mbox{respectively}), \]
that is, the group of isotopy classes of
elements of
$\Homeo (M , N)$ ($\Homeo (M ~\rel~ N)$, respectively),
where the isotopies are required to preserve $N$ setwise 
(to preserve $N$ pointwise, respectively),
a {\it mapping class group} and denoted it by
\[\MCG (M , N) ~(\MCG (M ~\rel~ N),~\mbox{respectively}).\]
The ``plus" subscripts, for instance in
$\Homeo_+ (M, N)$
and
$\MCG_+ (M, N)$,
indicate the subgroups of
$\Homeo (M, N )$ and
$\MCG (M, N )$ consisting of
orientation-preserving homeomorphisms and their isotopy classes, respectively. 	

\vspace{1em}

By a {\it graph}
we shall mean the underlying
space of a finite connected simplicial complex of dimension one. 
A point $x$ in a graph is called
a {\it vertex} if
$x$ does not have an open neighborhood that
is homeomorphic to
an open interval.
We denote by $v(G)$ the set of all
vertices of a graph $G$.
Throughout the paper a graph is always assumed not to have
any valency one vertices, i.e. any vertex admits no
open neighborhood homeomorphic to $[0,1)$.
The closure of each component of
$G \setminus v(G)$ is called
an {\it edge}. 
An edge $e$ of a graph $\Gamma$ is called a {\it cut edge} if 
$\Gamma \setminus \Int \thinspace e$ is disconnected. 
We note that the above definition of a graph 
allows multiple edges between two vertices, 
and an edge from a vertex to itself (i.e. a {\it loop}), and
the above definition of a vertex excludes vertices of valency two. 
These definitions are different
from the usual definitions of a graph and a vertex.

A {\it spatial graph} $\Gamma$ is a graph embedded
in a $3$-manifold $M$.
Two spatial graphs are said to be equivalent
if one can be transformed into the other by
an ambient isotopy of the $3$-manifold.
Note that a knot $K$ in $M$ is also a spatial graph.

For a spatial graph $\Gamma$ in $S^3$, the group
$\MCG(S^3, \Gamma)$ describes the symmetries of $\Gamma$.
For a knot $K$ in $S^3$, the group
$\MCG(S^3, K)$ is called the {\it symmetry group} of $K$,
see \cite{Kaw96}.

In \cite{Kod}, the following is proved:
\begin{theorem}[\cite{Kod}]
\label{thm: symmetry group of a graph is finitely presented}
For a spatial graph $\Gamma$ in $S^3$,
the group
$\MCG (S^3 , \Gamma)$
is finitely presented.
\end{theorem}

\subsection{Topological symmetry groups of graphs}
\label{subsec:Topological symmetry groups of graphs}

Let $\Gamma$ be a graph.
Let $X$ be a 1-dimensional simplicial complex such that $\Gamma = |X|$.
We denote by $\Aut (\Gamma)$ the group of
all simplicial automorphisms of the simplicial complex $X$. It is clear that
the group $\Aut (\Gamma)$ does not depend on the choice of $X$.
Let $\Gamma$ be a spatial graph in $S^3$.
The {\it topological symmetry group} $\TSG (S^3, \Gamma)$ and
the {\it positive topological symmetry group} $\TSG_+ (S^3, \Gamma)$
of the spatial graph $\Gamma$ in $S^3$
are subgroups of $\Aut(\Gamma)$ defined as follows:
\begin{eqnarray*}
\TSG(S^3, \Gamma) = \{
f \in \Aut(\Gamma) \mid \mbox{there exists }\tilde{f} \in
\Homeo(S^3, \Gamma) \mbox{ such that } \tilde{f}|_{\Gamma} \mbox{ induces }  f 
\},
\end{eqnarray*}
\begin{eqnarray*}
 \TSG_+(S^3, \Gamma) = \{
f \in \Aut(\Gamma) \mid \mbox{there exists }\tilde{f} \in
\Homeo_+(S^3, \Gamma) \mbox{ such that } \tilde{f}|_{\Gamma} \mbox{  induces } f 
\}.
\end{eqnarray*}
These groups are
originally defined by Simon \cite{Sim86}. 
See \cite{FNPT05, FT11} for details.
Obviously, the group $\TSG (S^3, \Gamma)$ is
a finite group.

The following proposition is a straightforward 
consequence of the definitions.

\begin{lemma}
\cite{Kod}
\label{pro:symmetry groups and topological symmetry groups}
Let $\Gamma$ be a spatial graph in $S^3$.
There is an exact sequence
\[ 1 \to \MCG(S^3 ~\rel~ \Gamma) \to \MCG (S^3, \Gamma) \to
\TSG (S^3, \Gamma) \to 1 \]
\[( 1 \to \MCG_+(S^3 ~\rel~ \Gamma) \to \MCG_+(S^3, \Gamma) \to
\TSG_+ (S^3, \Gamma) \to 1,
respectively.)\]
Hence $\MCG(S^3, \Gamma) \cong \TSG(S^3, \Gamma)$
$(\MCG_+(S^3, \Gamma) \cong \TSG_+(S^3, \Gamma)$, respectively$)$
if and only if
$\MCG(S^3 ~\rel~ \Gamma) \cong 1$
$(\MCG_+(S^3 ~\rel~ \Gamma) \cong 1$, respectively$)$.
\end{lemma}

By this lemma, to answer the question in the Introduction,
it is enough to determine when $\MCG(S^3 ~\rel~ \Gamma)$
(or $\MCG_+ (S^3 ~\rel~ \Gamma)$) is trivial.

\subsection{Review on boundary patterns}
\label{subsec:Review on boundary patterns}

In this subsection, we review the notion of
boundary pattern defined in
\cite{Joh79} and \cite{McC91}.
Let $M$ be a compact $3$-manifold.
A {\it boundary pattern} for $M$ consists of a set
$\underline{\underline{m}}$ of compact connected surfaces in
$\partial M$, such that the intersection of any $i$ of them 
is empty or consists of $(3 - i)$-manifolds.
A boundary pattern is said to be {\it complete} when
$\bigcup_{B \in \underline{\underline{m}}} B = \partial M$.

A boundary pattern $\underline{\underline{m}}$
of a $3$-manifold $M$ is said to be {\it useful}
if any disk $D$ properly embedded in $M$, where
$\partial D$ intersects $\partial B$
transversely for each $B \in \underline{\underline{m}}$ and
$\# (D \cap ( \bigcup_{B \in \underline{\underline{m}}}
\partial B)) \leqslant 3$,
bounds a disk $E$ in $\partial M$ such that
$E \cap ( \bigcup_{B \in \underline{\underline{m}}} \partial B )$
is the cone on
$\partial D \cap ( \bigcup_{B
\in \underline{\underline{m}}} \partial B ) $.

A $3$-manifold $M$ with  a
complete boundary pattern $\underline{\underline{m}}$ 
is said to be {\it simple} if it satisfies the following
three conditions:
\begin{enumerate}
\item
$M$ is irreducible and each element $B$ of $\underline{\underline{m}}$
is incompressible, 
\item
Every incompressible torus in $M$ is
parallel to an element of $\underline{\underline{m}}$ which is a torus, and 
\item
Any annulus $A$ in $M$ with $\partial A \cap
(\bigcup_{B \in \underline{\underline{m}}} \partial B ) = \emptyset$
is either compressible or parallel into 
an element $B$ of $\underline{\underline{m}}$.
\end{enumerate}


The mapping class group
of a manifold $M$ with a boundary pattern $\underline{\underline{m}}$,
denoted by $\MCG(M , \underline{\underline{m}})$,
is the group
$\MCG (M , B_1 , B_2 , \ldots , B_k)$,
if $\underline{\underline{m}} = \{ B_1 , B_2 , \ldots , B_k \}$.

\begin{theorem}[\cite{Joh79} Proposition 27.1]
\label{thm:Joh79 Prop 27.1}
Let $(M, \underline{\underline{m}})$ be a
simple $3$-manifold with complete and useful boundary pattern.
Then $\MCG(M, \underline{\underline{m}})$ is finite.
\end{theorem}

\section{Positive topological symmetry groups and
positive mapping class groups}
\label{sec:Positive topological symmetry groups and
positive mapping class groups}

Let $V$ be a handlebody and let $D_1 , D_2 , \ldots , D_n$ be 
mutually disjoint, mutually non-parallel 
essential disks in $V$. 
Suppose that $V \setminus ( \bigcup_{i=1}^n \Int \thinspace N(D_i ; V) )$ 
consists of 3-balls. 
Then there exists a graph $\Gamma$ embedded in $\Int \thinspace V$ 
such that 
\begin{enumerate}
\item
$\Gamma$ is a deformation retract of $V$; 
\item
$\Gamma$ intersects $\bigcup_{i=1}^n D_i$ 
transversely at the points in the interior of the 
edges of $\Gamma$; and   
\item
$\Gamma$ has exactly $n$ edges 
$e_1 , e_2 , \ldots , e_n$ and 
$ \# ( e_i \cap D_j ) = \delta_{ij}$, where 
$\delta_{ij}$ is the Kronecker delta. 
\end{enumerate}
We call the above graph $\Gamma$ 
the {\it dual graph} of $\{D_1 , D_2 , \ldots , D_n \}$. 
Note that such a graph is uniquely determined up to isotopy. 
\begin{lemma}[\cite{Kod}]
\label{lem:triviality of a mapping class group}
Let $V$ be a handlebody and let $\Gamma$ 
be a graph embedded in $V$ which is a deformation retract of $V$. 
Then the group $\MCG (V , \Gamma ~\rel~ \partial V)$ 
is trivial.  
\end{lemma}

Let $\Gamma$ be a spatial graph in $S^3$. 
Set $V = N(\Gamma)$. 
Let 
\[
\Delta = \{ D_1 , D_2 , \ldots , D_{n_1}, 
D_1^\prime , D_2^\prime , \ldots , D_{n_2}^\prime, 
D_1^{\prime \prime} , D_2^{\prime \prime} , \ldots , D_{n_3}^{\prime\prime} 
 \} 
\] 
be the family of 
essential disks in $V$ such that 
\begin{itemize}
\item 
$\Gamma$ is a dual graph of $\Delta$, 
\item
the set $\{D_1, D_2, \ldots , D_{n_1}\}$ corresponds to the set of 
loops of $\Gamma$, 
\item
the set $\{D_1^\prime, D_2^\prime, \ldots , D_{n_2}^\prime\}$ 
corresponds to the set of cut edges of $\Gamma$. 
\item
the set $\{D_1^{\prime \prime}, D_2^{\prime \prime}, \ldots , 
D_{n_3}^{\prime \prime}\}$ 
corresponds to the set of non-loop, non-cut edges of $\Gamma$. 
\end{itemize}
For each $1 \leqslant i \leqslant n_1$, 
we take parallel copies 
$D_{i,1}$, $D_{i,2}$, $D_{i,3}$, $D_{i,4}$ of $D_i$, and 
for each $1 \leqslant i \leqslant n_3$, 
we take parallel copies 
$D_{i,1}^{\prime \prime}$, $D_{i,2}^{\prime \prime}$ 
of $D_i^{\prime \prime}$ 
so that all the disks $D_{i,j}$, $D^\prime$, $D^{\prime \prime}_{i,j}$ 
are mutually disjoint.  
Denote by $\Delta^\prime$ the collection of disks 
\[ 
\{ D_{i, j} \mid 1 \leqslant i \leqslant n_1, j= 1,2,3,4 \}
\cup 
\{ D_1^\prime, D_2^\prime, \ldots, D_{n_1}^\prime \} 
\cup \{ D_{i, j}^{\prime \prime} \mid 
1 \leqslant i \leqslant n_3, j= 1,2 \} .
\] 
Let $\underline{\underline{m(\Gamma)}}$ be the set of 
the closure of each component of 
$\partial N(\Gamma) \setminus \bigcup_{D \in \Delta^\prime} \partial D$.
Then $\underline{\underline{m (\Gamma)}}$ 
is a complete and useful boundary-pattern of 
$E(\Gamma) = E(V)$, see Figure \ref{fig:graph_and_boundary-pattern}. 
\begin{figure}[htbp]
\begin{center}
\includegraphics[width=12cm,clip]{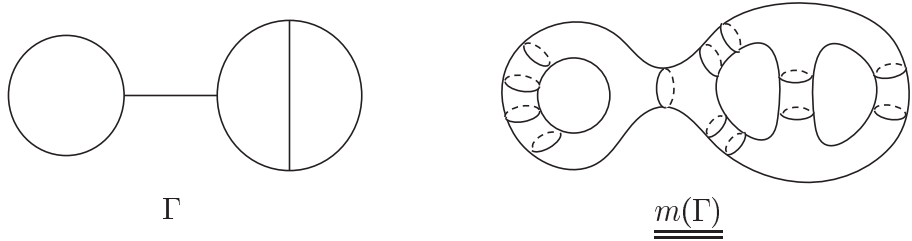}
\caption{A spatial graph $\Gamma$ and the corresponding 
boundary-pattern $\underline{\underline{m(\Gamma)}}$.}
\label{fig:graph_and_boundary-pattern}
\end{center}
\end{figure}

\begin{lemma}
Let $\Gamma \subset S^3$ be a spatial graph.
Then we have
\begin{eqnarray*}
\MCG(S^3, \Gamma) &\cong& \MCG(E(\Gamma) ,  \underline{\underline{ m (\Gamma)}})
, \\
\MCG_+(S^3, \Gamma) &\cong& \MCG_+(E(\Gamma) , \underline{\underline{ m (\Gamma)}})
.
\end{eqnarray*}
\end{lemma}
\begin{proof}
Set $V = N(\Gamma)$. 
Let $\Delta^\prime$ 
be the collection of 
essential disks in $V$ defined as in the above. 
By Lemma \ref{lem:triviality of a mapping class group}, 
we see at once that  
$\MCG (S^3, \Gamma) 
\cong \MCG (S^3 , V, \cup_{D \in \Delta^\prime} D)$ 
($\MCG_+ (S^3, \Gamma) 
\cong \MCG_+ (S^3 , V, \cup_{D \in \Delta^\prime} D)$, respectively) and 
$\MCG (S^3 , V, \cup_{D \in \Delta^\prime} D)
\cong \MCG (E(\Gamma) , \underline{\underline{m(\Gamma)}})$ 
($\MCG_+ (S^3 , V, \cup_{D \in \Delta^\prime} D)
\cong \MCG_+ (E(\Gamma) , \underline{\underline{m(\Gamma)}})$, respectively) 
using the Alexander's trick. 
Hence the assertion follows. 
\end{proof}

By a {\it handlebody-knot},
we shall mean a handlebody $V$ embedded in
the 3-sphere. 
{We note} that the exterior 
$E(V)$ of $V$ 
is always irreducible.
We recall that a 3-manifold is said to be {\it atoroidal} if
it does not contain an embedded, non-boundary parallel,
incompressible torus.

\begin{lemma}
\label{lem:MCG+(E(V) rel partial E(V)) = 1}
Let $V \subset S^3$ be a handlebody-knot of genus
at least two.
Assume that $E(V)$ is atoroidal.
Then $\MCG_+(E(V) ~ \rel ~ \partial E(V)) = 1$.
\end{lemma}
\begin{proof}
Let $\Gamma$ be a spatial graph in $S^3$ such that
$N(\Gamma) = V$
and that $(E(V) , \underline{\underline{ m }}) $ is simple,
where
we set $\underline{\underline{ m }} =
 \underline{\underline{ m (\Gamma)}}$.
Such a graph can be taken, for instance, as follows:
Let $g$ be the genus of $V$.
If $g = 2$, take $\Gamma$ as a $\theta$-curve, i.e.
a graph on two vertices with three edges
joining them.
If $g > 2$, take $\Gamma$ as a {\it wheel graph} with $g+1$ vertices,
i.e. the 1-skeleton of a $g$-gonal pyramid.
By Theorem \ref{thm:Joh79 Prop 27.1},
$\MCG_+ (E(V) , \underline{\underline{ m }}) $ is finite. 

Consider the following exact sequence
corresponding to the fibration
$\Homeo_+ (E(V)) \to
\Homeo_+ (\partial E(V)) $:
\begin{eqnarray*}
\cdots &\to& \pi_1( \Homeo_+ (\partial E(V))) \to
\MCG_+ ( E(V) \thinspace \rel \thinspace \partial E(V))
\to \MCG_+ (E(V)) \\
&\to& \MCG_+ (\partial E(V)) \to 0
\end{eqnarray*}
Since $\Gamma$ is not a knot, the genus of
$\partial E(V)$ is at least two.
Hence we have
\[\pi_1( \Homeo_+ (\partial E(V))) = 1\]
by \cite{Ham62, Ham65, Ham66}.
It follows that the map
$\MCG_+ (\partial E(V) \thinspace \rel \thinspace \partial E(V))
\to \MCG_+ (E(V))$ in the
above sequence is an injection.
On the other hand, the sequence of inclusions
\[ \Homeo_+ (E(V) ~\rel~ \partial E(V) ) \subset
 \Homeo_+ (E(V) , \underline{\underline{ m }})
\subset \Homeo_+ ( E(V))\]
induces the following sequence of homomorphisms:
\[ \MCG_+ (E(V) ~\rel~ \partial E(V) ) \to
\MCG_+ (E(V) , \underline{\underline{ m }}) \to \MCG_+ ( E(V)) . \]
Since $\MCG_+ (E(V) ~\rel~ \partial E(V) ) \to \MCG_+ ( E(V))$ is
an injection and $\MCG_+ (E(V) , \underline{\underline{ m }})$ is finite,
the group $\MCG_+ (E(V) ~\rel~ \partial E(V) )$ is also finite.
By \cite{HT87} (see also \cite{HM97}), the group
$\MCG_+ (E(V) ~\rel~ \partial E(V) ) $ is torsion free,
hence $\MCG_+ (E(V) ~\rel~ \partial E(V) ) = 1$.
\end{proof}

\begin{remark}
In the above proof, we use results
for the homotopy groups of
the automorphism groups of 2-manifolds and 3-manifolds in
the topological and smooth categories while we are
working in the piecewise linear category.
However, in dimensions at most 3,
it has been known that the information on the homotopy types of automorphism groups
of manifolds is the same to the smooth, topological, and piecewise linear categories,
see e.g. \cite{Ham74}.
\end{remark}

Since the group $MCG(V ~\rel~ \partial V)$ is trivial, for a handlebody $V$,
the above lemma immediately implies the following:
\begin{corollary}
Let $V$ be a handlebody-knot in the $3$-sphere whose genus is at least two.
Then $\MCG_{+}(S^3 ~ \rel ~ V) \cong 1$
if and only if $E(V)$ is atoroidal.
\end{corollary}
This corollary implies that $\MCG_+(S^3, V)$ can
be regarded as a subgroup of the handlebody-group
$\MCG_+(V)$ when $E(V)$ is atoroidal.

\begin{definition}
Let $\Gamma$ be a spatial graph in $S^3$.
\begin{enumerate}
\item
Let $P$ be a 2-sphere embedded in $S^3$ satisfying:
\begin{itemize}
\item
the sphere $P$ intersects $\Gamma$ in a single vertex, and
\item
each of the two
components of $S^3 \setminus P$ contains
non-empty part of $\Gamma$.
\end{itemize}
Then $P$ is called a {\it type I sphere} for $\Gamma$. See the left in
Figure \ref{fig:type_I_II_III_spheres}.
\item
Let $P$ be a 2-sphere embedded in $S^3$ satisfying:
\begin{itemize}
\item
the sphere $P$ intersects $\Gamma$ in exactly two vertices,
\item
the closure of neither component of
$(S^3 \setminus P) \cap \Gamma$ is a single point
or a
single edge, and
\item
the annulus $P \cap E(\Gamma)$
is incompressible in
$E(\Gamma)$.
\end{itemize}
Then $P$ is called a {\it type II sphere} for $\Gamma$. See the middle in
Figure \ref{fig:type_I_II_III_spheres}.
\item
A 2-sphere with two points identified to a single point is called a {\it pinched sphere} and the identified point is called its pinch point. 
Let $P$ be a pinched sphere in $S^3$ with
a pinch point $p$ satisfying:
\begin{itemize}
\item
The pinch point $p$ is a vertex of $\Gamma$
such that $P \cap \Gamma = \{ p \}$,
\item
the closure of neither component of
$(S^3 \setminus P) \cap \Gamma$ is a single point nor a single loop, and
\item
the annulus
$P \cap E(\Gamma)$
is incompressible in $E(\Gamma)$.
\end{itemize}
Then $\Gamma$ is called a {\it type III sphere} for $\Gamma$.
See the right in
Figure \ref{fig:type_I_II_III_spheres}.
\end{enumerate}
\begin{figure*}[htbp]
\begin{center}
\includegraphics[width=14cm,clip]{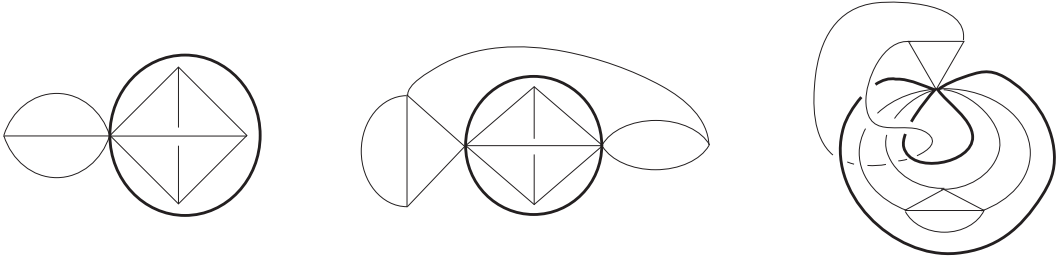}
\caption{Type I, II and III spheres}
\label{fig:type_I_II_III_spheres}
\end{center}
\end{figure*}
\end{definition}

\begin{theorem}
\label{thm:the mcg and the positive tsg}
Let $\Gamma$ be a spatial graph in $S^3$ which is not a knot.
Then $\MCG_+(S^3, \Gamma) \cong \TSG_+(S^3, \Gamma)$
if and only if $E(\Gamma)$ is atoroidal and
$\Gamma$ admits neither
Type I, II nor III spheres.
\end{theorem}
\begin{proof}
The ``only if" part is straightforward.

Let $\Gamma$ be a spatial graph in $S^3$ such that
 $E(\Gamma)$ is atoroidal and
$\Gamma$ admits neither
Type I, II nor III spheres.
By Lemma \ref{pro:symmetry groups and topological symmetry groups},
it suffices to show that
$\MCG_+(S^3 ~ \rel ~ \Gamma) = 1$.
Set $\underline{\underline{ m }} = \underline{\underline{ m (\Gamma)}}$.
Recall that $\MCG_+ (E( \Gamma ) , \underline{\underline{ m }}) $ is a finite group
by Theorem \ref{thm:Joh79 Prop 27.1}.

Let $\mathcal{A} = \{ A_1, A_2, \ldots, A_{n^\prime} \}$
($\mathcal{B} = \{ B_1, B_2, \ldots, B_{m} \}$, respectively) be the
set of annulus components (non-annulus components, respectively)
of $\underline{\underline{ m }}$.
We note that each $B_i$ is a planar surface.
Set $\partial B_i = C_{i_1} \cup C_{i_2} \cup \cdots \cup C_{i_{n_i}}$.

Suppose that $\MCG_+ (E( \Gamma ) , \underline{\underline{ m }})$ is not trivial.
Then there exists an  element $f \in \Homeo_+
(E( \Gamma ) , \underline{\underline{ m }})$
that is not isotopic to the identity.
The proof is divided into two cases.

\vspace{1em}

\noindent{\it Case} 1: For each $B_i \in \mathcal{B}$,
$f | _{B_i} $ is trivial as an element of
$\MCG_+ (B_i , C_{i_1} , C_{i_2} , \ldots , C_{i_{n_i}})$.
We may assume that $f | _{\bigcup_{i=1}^m B_i}$ is the identity.
Then there exists a set of pairwise disjoint, pairwise non-parallel
essential simple closed curves
$\gamma_1, \gamma_2 , \ldots , \gamma_l$ in $\bigcup_{i=1}^{n^\prime} A_i$
such that $f | _{\partial E(\Gamma)}$ is
$\prod_{i=1}^l {\tau_{\gamma_i}}^{\alpha_i}$
as an element of $\MCG_+(\partial E(\Gamma))$,
where $\tau_{\gamma_i}$ is a 
Dehn twist along $\gamma_i$, and $\alpha_i$ is a non-zero integer.
By Lemma \ref{lem:MCG+(E(V) rel partial E(V)) = 1},
$\prod_{i=1}^l {\tau_{\gamma_i}}^{\alpha_i} \neq 1 \in \MCG_+(\partial E(\Gamma))$.
It follows that the order of $f | _{\partial E(\Gamma)}$
as an element of $\MCG(\partial E(\Gamma))$ is infinite since
each $\gamma$ is a boundary of a meridian disk of
$N(\Gamma)$.
This contradicts the fact that
$\MCG_+ (E( \Gamma ) , \underline{\underline{ m }}) $ is a finite group.

\vspace{1em}

\noindent{\it Case} 2:
There exists an element $B_i \in \mathcal{B}$ such that
$f | _{B_i} $ is not trivial as an element of
$\MCG_+ (B_i , C_{i_1} , C_{i_2} , \ldots , C_{i_{n_i}})$.
Since the group
$\MCG_+ (B_i , C_{i_1} , C_{i_2} , \ldots , C_{i_{n_i}})$
is torsion free
by \cite{FN62},
this contradicts, again, the fact that
$\MCG_+ (E( \Gamma ) , \underline{\underline{ m }}) $ is a finite group.
\end{proof}

Let $\Gamma \subset S^3$ be a spatial graph.
If $\Gamma$ admits either Type I, II or III
spheres, then
$\MCG_+(S^3, \Gamma)$ is not a finite group
since $\MCG_+(E(\Gamma), \underline{\underline{ m (\Gamma)}})$
admits non-trivial twists along essential
disks or annuli corresponding to the spheres.
Hence we have the following:

\begin{corollary}
\label{cor:finiteness of a mapping class group}
Let $\Gamma$ be a spatial graph in $S^3$.
Then $\MCG_+(S^3, \Gamma)$ is a finite group if and only if
$\MCG_+(S^3, \Gamma) \cong \TSG_+(S^3, \Gamma)$.
\end{corollary}

It follows from Theorem \ref{thm:the mcg and the positive tsg} 
and \cite{FNPT05} that 
the group $\MCG_+ (S^3, \Gamma)$ is a finite subgroup of $SO(4)$  
when $E(\Gamma)$ is atoroidal and
$\Gamma$ admits neither
Type I, II nor III spheres.

Let $k$ be a natural number.
Recall that a graph $\Gamma$ is said to be
{\it $k$-connected} if
there does not exist a set
$\{v_1, v_2, \ldots, v_{k-1} \}$
of $k-1$ vertices of $\Gamma$ such that
$\Gamma \setminus \{v_1, v_2, \ldots, v_{k-1} \}$
is not connected as a topological space.

The following corollary immediately
follows from Theorem \ref{thm:the mcg and the positive tsg}:
\begin{corollary}
\label{cor:3-connected graph}
Let $\Gamma$ be a $3$-connected graph. 
Then an embedding of $\Gamma$ into $S^3$ satisfies 
$\MCG_+(S^3, \Gamma) \cong \TSG_+(S^3, \Gamma)$
if and only if $\Gamma$ is atoroidal.
\end{corollary}

We note that even when $\Gamma$ is not 3-connected, 
most embeddings of $\Gamma$ into $S^3$
satisfy $\MCG_+(S^3, \Gamma) \cong \TSG_+(S^3, \Gamma)$,
see Section \ref{sec:Easy examples}.

\section{Topological symmetry groups and mapping class groups}
\label{sec:Topological symmetry groups and
positive mapping class groups}

In the previous section, we discussed a
topological condition for a spatial graph $\Gamma \subset S^3$
so that the positive mapping class group $\MCG_+(S^3, \Gamma)$
is isomorphic to the positive topological symmetry group
$\TSG_+(S^3, \Gamma)$.
Of course, even when $\MCG_+(S^3, \Gamma)$
is isomorphic to $\TSG_+(S^3, \Gamma)$,
$\MCG(S^3, \Gamma)$ might differ from
$\TSG(S^3, \Gamma)$.
A trivial example is the case of a spatial
3-connected graph $\Gamma \subset S^3$ contained in
an embedded 2-sphere $S^2$ in $S^3$.
In this case, there exists a reflection $f$ through
the 2-sphere.
Then $f$ fixes the 2-sphere pointwise and
in particular it restricts to an identity on $\Gamma$.
Obviously, $f$ is orientation-reversing and
hence $\MCG(S^3 ~\rel~ \Gamma) \cong \Integer / 2 \Integer$
while $\MCG_+ (S^3 ~\rel~ \Gamma) \cong 1$.
By Lemma \ref{pro:symmetry groups and topological symmetry groups},
this implies that
$\MCG(S^3, \Gamma) \ncong \TSG(S^3, \Gamma)$
while $\MCG_+(S^3, \Gamma) \cong \TSG_+(S^3, \Gamma)$.
In this section, we prove that
this is the only case.

%

\begin{proposition}
\label{prop:involution and realization problem}
Let $\Gamma$ be a spatial graph in $S^3$ such that
$\Gamma$ is not a knot.
Let $h$ be an orientation reversing
homeomorphism
of $S^3$ that fixes $\Gamma$ pointwise such that
$h^2 \in \Homeo_+ (S^3 ~ \rel ~ \Gamma)$
is isotopic $($rel $\Gamma)$ to the identity.
Then there exists a homeomorphism
$f  \in \Homeo_+ (S^3 ~ \rel ~ \Gamma)$
such that
$h$ is isotopic $($rel $\Gamma)$ to
$f$ and
$f^2 = \id$.
\end{proposition}
\begin{proof}
By an isotopy (rel $\Gamma$) we may assume that
$h ( N(\Gamma)) = N(\Gamma)$.
Since $E(\Gamma)$ is not a Seifert fibered 3-manifold,
it follows from
\cite{HT78, HT83} that $h$ can be isotoped (rel $\Gamma$)
to a map $h_1: (S^3, \Gamma) \to (S^3, \Gamma)$ with
$h_1 ( E(\Gamma) ) = E(\Gamma)$ and
$(h_1 |_{E(\Gamma)})^2 = \id_{E(\Gamma)}$.
Let $e_1, e_2 , \ldots, e_n$ be the edges of $\Gamma$.
Using a standard argument of
Riemannian geometry,
it is easy to see that for each edge $e_i$ of $\Gamma$
there exists a meridian disk $D_i$ of $N(\Gamma)$ such that
$h_1 ( \partial D_i ) = \partial D_i$ and
$D_i$ intersects $\Gamma$ once and transversely at a point $p_i$
in $\Int \thinspace e_i$.
Note that we have $h_1|_{\partial D_i} = id_{D_i}$ for each $i$.
Since $D_i$ and $h_1(D_i)$ are parallel in $N(\Gamma)$,
$h_1$ can be isotoped  (rel $E(\Gamma) \cup \Gamma$) to
a map $h_2: (S^3, \Gamma) \to (S^3, \Gamma)$ with
$h_2 (D_i) = D_i $ for each $i$.
Then by the Alexander's trick, we may isotope $h_2$ (rel $E(\Gamma) \cup \Gamma$)
preserving each meridian disk $D_i$ as a set
to a map $h_3: (S^3, \Gamma) \to (S^3, \Gamma)$ with
$(h_3 |_{D_i}) ^2 = \id_{D_i}$ for each $i$.
Finally, let $B$ be the closure of
a component of $N(\Gamma) \setminus \bigcup_i D_i$.
since $B \cap \Gamma$ is the cone over points in $\partial B$ with vertex
at the center of $B$, we may isotope $h_3$ (rel $E(B)$)
to a map $h_4$ with $(h_4 |_B) ^2 = \id_B$.
Performing this isotopy for the closure of
each component of $N(\Gamma) \setminus \bigcup_i D_i$,
$h_4$ can be isotoped (rel $E(\Gamma) \cup \Gamma \cup \bigcup_i D_i$)
to a required orientation-reversing involution
$f: (S^3, \Gamma) \to (S^3, \Gamma)$.
\end{proof}

A spatial graph $\Gamma \subset S^3$
is called a {\it plane graph}
if $\Gamma$ is embedded in a sphere $S^2$ in $S^3$.

\begin{theorem}
\label{thm:involution and the symmetry group of a spatial graph}
Let $\Gamma$ be a spatial graph in $S^3$ such that
$\Gamma$ is not a knot and that
$\MCG_+(S^3, \Gamma) \cong \TSG_+ (S^3, \Gamma)$.
Then $\MCG(S^3, \Gamma) \cong \TSG (S^3, \Gamma)$
if and only if $\Gamma$ is not a plane graph.
\end{theorem}
\begin{proof}
The ``only if" part is clear.

Assume that $\MCG(S^3, \Gamma) \ncong \TSG (S^3, \Gamma)$.
It is equivalent to say that
$\MCG(S^3 ~ \rel ~ \Gamma) \neq 1$,
and hence $\MCG(S^3 ~\rel~\Gamma) \cong \Integer / 2 \Integer$.
Then there exists an orientation-reversing
homeomorphism $h \in  \Homeo (S^3 ~ \rel ~ \Gamma)$
such that 
 $h^2 \in \Homeo_+ (S^3 ~ \rel ~ \Gamma)$
is isotopic (rel $\Gamma$) to the identity.

By Proposition \ref{prop:involution and realization problem},
$h$ can be isotoped (rel $\Gamma$)
to an orientation-reversing involution
$f: (S^3, \Gamma) \to (S^3, \Gamma)$
with
$\Gamma \subset \Fix(f)$, here
$\Fix(f)$ for a map $f: S^3 \to S^3$ is the
set of fixed points of $f$.
Since the fixed point set of
an orientation reversing involution of $S^3$
is either a 2-sphere or two points by 
Smith theory \cite{Smi42} (see also \cite{{HS59}}),
$\Fix (f)$ is the 2-sphere.
This implies that $\Gamma$ is a plane graph.
%

%
%
\end{proof}

\begin{remark}
Theorem
\ref{thm:involution and the symmetry group of a spatial graph}
implies that
$\MCG(S^3 ~ \rel ~ \Gamma) \neq \MCG_+ (S^3 ~ \rel ~ \Gamma)$
if and only if $\Gamma$ is a plane graph under the condition
$\MCG_+(S^3 ~\rel~ \Gamma) \cong 1$.
Finding a general condition which detects
when $\MCG(S^3, \Gamma)$ differs from $\MCG_+ (S^3, \Gamma)$ is
another interesting problem. See e.g. \cite{NT09}. 

Let $\Gamma$ be the spatial graph depicted in Figure
\ref{fig:figure_eight_theta}.
\begin{figure*}[htbp]
\begin{center}
\includegraphics[width=2.5cm,clip]{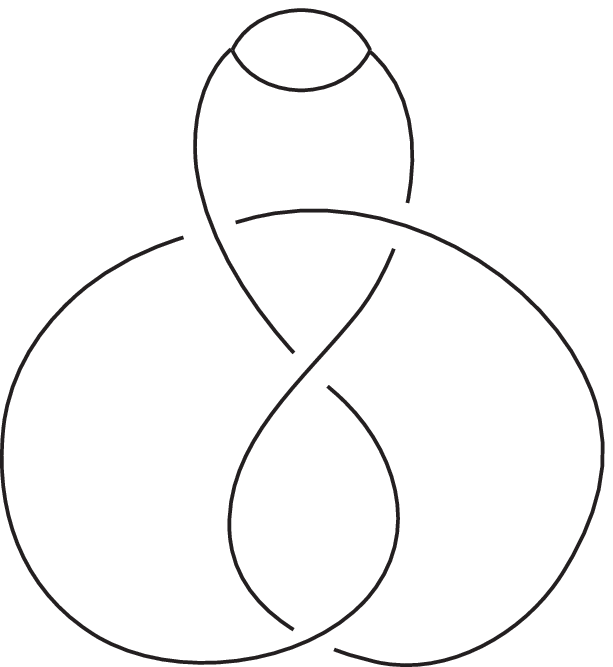}
\caption{}
\label{fig:figure_eight_theta}
\end{center}
\end{figure*}
It is easily seen that
$\Gamma$ is not a plane graph and
$\MCG(S^3, \Gamma)$ differs from $\MCG_+ (S^3, \Gamma)$.
\end{remark}

\section{Easy examples}
\label{sec:Easy examples}
One of the simple but important examples of 
spatial graphs in $S^3$ is a tunnel number one 
knot or link together with a specific tunnel attached.
A tunnel number one knot or link in $S^3$ is a knot 
or a link $K$ with an arc $\tau$ (called a tunnel for $K$) 
such that $K \cap \tau = K \cap \partial \tau$ 
and $E(K \cup \tau)$ is a genus two handlebody.
Allowing that $\partial \tau$ could be a single point 
(that is, a tunnel could be a circle rather than an arc), 
the spatial graph $\Gamma = K \cup \tau$ is either 
a $\theta$-curve or a bouquet of two circles if $K$ is a knot, 
or a handcuff if $K$ is a link 
(a tunnel number one link necessarily consists of at most two components).

In this section, we give a complete list 
of the groups $\MCG_+(S^3, \Gamma)$, $\MCG(S^3, \Gamma)$, 
$\TSG_+(S^3, \Gamma)$ and $\TSG(S^3, \Gamma)$ 
for every spatial graph $\Gamma$ which is the union of 
a  tunnel number one knot or link $K$ and its tunnel $\tau$.
The following lemma is just a direct translation of 
Proposition 17.2 in \cite{CM} into our cases:

\begin{lemma}
\label{lem:stablizer subgroups}
Let $\Gamma$ be the union of a nontrivial tunnel number one 
knot or link $K$ and its tunnel $\tau$ in $S^3$, 
which is either a $\theta$-curve or a handcuff.
\begin{enumerate}
\item If $K$ is a Hopf link and $\tau$ is 
its upper or lower tunnel when $K$ 
is considered as a $2$-bridge link, then 
$\MCG(S^3, \Gamma)$ is isomorphic to 
the dihedral group $D_4$ of order $8$.
\item If $K$ is a nontrivial $2$-bridge knot or link 
which is not a Hope link and $\tau$ is its upper or lower tunnel, 
then $\MCG(S^3, \Gamma)$ is isomorphic to 
$\Integer / 2 \Integer \times \Integer / 2 \Integer$.
\item Otherwise, $\MCG(S^3, \Gamma)$ 
is isomorphic to $\Integer / 2 \Integer$.
\end{enumerate}
\end{lemma}

From Lemma \ref{lem:stablizer subgroups} and
Theorems \ref{thm:the mcg and the positive tsg}
and \ref{thm:involution and the symmetry group of a spatial graph}
with some computations, the following proposition is immediate.
We note that the plane handcuff and the plane
bouquet admit the spheres of type I,
while the others considered in the following proposition
do not.

\begin{proposition}
\label{prop:examples}
Let $\Gamma$ be the union of a tunnel number one knot or link 
and its tunnel in $S^3$.

\begin{enumerate}

\item Let $\Gamma$ be a $\theta$-curve.

\begin{enumerate}

\smallskip
\item If $\Gamma$ is a plane graph,

\smallskip
$\MCG_+(S^3, \Gamma) \cong
\TSG_+(S^3, \Gamma) =
\TSG(S^3, \Gamma) =
\Aut (\Gamma) \cong \Integer / 2 \Integer \times
D_3,$

\smallskip
$\MCG(S^3, \Gamma) \cong \Integer / 2 \Integer \ltimes \MCG_+(S^3, \Gamma)$,

\smallskip
where $D_3$ is the dihedral group of order $6$,

\smallskip
\item If $\Gamma$ is the union of a nontrivial $2$-bridge knot and its upper or lower tunnel,

\smallskip
$\MCG_+(S^3, \Gamma) =
\MCG(S^3, \Gamma) \cong
\TSG_+(S^3, \Gamma) =
\TSG(S^3, \Gamma) \\
\cong \Integer / 2 \Integer \times \Integer / 2 \Integer.
$

\smallskip
\item Otherwise,

\smallskip
$\MCG_+(S^3, \Gamma) =
\MCG(S^3, \Gamma) \cong
\TSG_+(S^3, \Gamma) =
\TSG(S^3, \Gamma) \cong \Integer / 2 \Integer.
$
\end{enumerate}

\smallskip
\item Let $\Gamma$ be a handcuff.

\smallskip
\begin{enumerate}
\item If $\Gamma$ is a plane graph,

\smallskip
$\MCG_+(S^3, \Gamma) \cong
\Integer / 2 \Integer \ltimes (\Integer / 2 \Integer \times \Integer)$,

\smallskip
$\MCG(S^3, \Gamma) \cong \Integer / 2 \Integer \ltimes \MCG_+(S^3, \Gamma)$,

\smallskip
$\TSG_+(S^3, \Gamma) =
\TSG(S^3, \Gamma) =
\Aut (\Gamma) \cong
\Integer / 2 \Integer \ltimes
( \Integer / 2 \Integer \times \Integer / 2 \Integer )$.

\smallskip
We note that in this case $\MCG_+(S^3, \Gamma)$ and 
$\MCG(S^3, \Gamma)$ are not finite while $\TSG_+(S^3, \Gamma)$ and 
$\TSG(S^3, \Gamma)$ are finite, of order $8$. 

\smallskip
\item If $\Gamma$ is the union of a Hopf link and its upper or lower tunnel,

\smallskip
$\MCG_+(S^3, \Gamma) \cong
\TSG_+(S^3, \Gamma) \cong
\Integer / 2 \Integer \times \Integer / 2 \Integer$,

\smallskip
$\MCG(S^3, \Gamma) \cong \TSG(S^3, \Gamma) \cong
\Integer / 2 \Integer \ltimes \MCG_+(S^3, \Gamma) \cong D_4.$

\smallskip
\item If $\Gamma$ is the union of a nontrivial $2$-bridge link, 
except the Hopf link, and its upper or lower tunnel,

\smallskip
$\MCG_+(S^3, \Gamma) =
\MCG(S^3, \Gamma)
\cong
\TSG_+(S^3, \Gamma) =
\TSG(S^3, \Gamma) \\
\cong
\Integer / 2 \Integer \times \Integer / 2 \Integer$.

\smallskip
\item Otherwise,

\smallskip
$\MCG_+(S^3, \Gamma) =
\MCG(S^3, \Gamma) \cong
\TSG_+(S^3, \Gamma) =
\TSG(S^3, \Gamma) \cong \Integer / 2 \Integer$.
\end{enumerate}

\smallskip
\item Finally, let $\Gamma$ be a spatial bouquet of two circles.

\smallskip
\begin{enumerate}

\smallskip
\item If $\Gamma$ is a plane graph,

\smallskip
$\MCG_+(S^3, \Gamma) \cong
\Integer / 2 \Integer \ltimes (\Integer / 2 \Integer \times \Integer)$,

\smallskip
$\MCG(S^3, \Gamma) = \Integer / 2 \Integer \ltimes \MCG_+(S^3, \Gamma)$,

\smallskip
$\TSG_+(S^3, \Gamma) =
\TSG(S^3, \Gamma) =
\Aut (\Gamma) \cong
\Integer / 2 \Integer \ltimes
( \Integer / 2 \Integer \times \Integer / 2 \Integer )$.

\smallskip
\item Otherwise,

\smallskip
$\MCG_+(S^3, \Gamma) =
\MCG(S^3, \Gamma) \cong
\TSG_+(S^3, \Gamma) =
\TSG(S^3, \Gamma) \cong \Integer / 2 \Integer$.

\end{enumerate}
\end{enumerate}
\end{proposition}

\noindent {\bf Acknowledgments.}
The authors wish to express their gratitude 
to Erica Flapan for posing the question and 
to Darryl McCullough for providing them
his very helpful idea on the topic.
Part of this work was carried out while the
second-named author was visiting
the Mathematisches Forschungsinstitut Oberwolfach.
The institute kindly offered the stay
while his University was affected
by the 2011 Tohoku earthquake.
He is grateful to the MFO and its staffs for
the offer, financial support and warm hospitality.
He would also like to thank the Leibniz Association for
travel support. 
Finally, the authors are grateful to the anonymous referee 
for his or her helpful comments that improved the presentation.

\end{document}